\documentclass[12pt]{article}
\usepackage{amssymb,amsfonts,amsmath,latexsym,epsf,url}
\usepackage[usenames,dvipsnames]{pstricks}
\usepackage{pstricks-add}
\usepackage{subfigure}
\usepackage{graphicx}
\usepackage{color,soul}
\usepackage{epsfig}
\usepackage{tikz}
\usepackage[colorlinks]{hyperref}
\usepackage{cite}
\usepackage{algorithm}
\usepackage{algpseudocode}
\usepackage{graphicx}
\usepackage{booktabs}
\usepackage{multirow}

\newtheorem{theorem}{Theorem}[section]

\newtheorem{lemma}[theorem]{Lemma}

\newtheorem{example}[theorem]{Example}

\begin{document}

\def\nt{\noindent}

\title{On the Spectrum of the Line Graph of a Family of Bipartite Graphs Arising from the Boolean Lattice}

\author{
	Ali Zafari$^1$\footnote{Corresponding author} \and	
Saeid Alikhani$^{2}$
}


\maketitle

\begin{center}

$^1$Department of Mathematics, Faculty of Science,
Payame Noor University, P.O. Box 19395-4697, Tehran, Iran\\ 
{\tt zafari.math@pnu.ac.ir}
\medskip

$^{2}$Department of Mathematical Sciences, Yazd University, 89195-741, Yazd, Iran\\
{\tt alikhani@yazd.ac.ir}

\end{center}
	
\begin{abstract}
The Boolean lattice $BL_n$, $n\geq 3$, is the graph whose vertex set is the collection of all subsets of $[n]=\{1,2,\ldots,n\}$, where two subsets $U$ and $W$ are adjacent if and only if their symmetric difference has precisely one element. In the graph $BL_n$, the \emph{layer} $L_k$ is the family of all $k$-element subsets of $[n]$. The subgraph $BL_n(k-1,k)$ is the induced subgraph of $BL_n$ on layers $L_{k-1}$ and $L_{k}$. This graph is bipartite and, when $n=2k-1$, is $k$-regular and isomorphic to the bipartite double cover $2{\cdot}O_k$ of the odd graph $O_k$. In this paper, we determine the full adjacency spectrum---eigenvalues together with their multiplicities---of the line graph $L(BL_n(k-1,k))$ for all admissible values of $n$ and $k$. As a consequence, we show that $L(BL_n(k-1,k))$ is an integral graph whenever $n = 2k-1$, and we recover as a special case the spectrum of the line graph $L(n)$ of $BL_n(1,2)$ established by Mirafzal~\cite{pap-sm-1}.
\end{abstract}


\medskip
\noindent{\bf AMS Subj.\ Class.:}  05C50, 05C25.

\section{Introduction}
\label{sec:introduction}

\subsection*{Background and motivation}

Spectral graph theory---the study of graphs through the eigenvalues and eigenvectors of their associated matrices---is one of the most active areas of algebraic combinatorics. The adjacency spectrum of a graph encodes a wealth of structural information: it governs expansion properties, random walk mixing times, the existence of perfect matchings, and chromatic properties, among many other things. A comprehensive account of the classical theory can be found in the monographs of Brouwer and Haemers~\cite{Brower-2} and Godsil and Royle~\cite{pap-cg-1}.

Among the most studied graph families in this context are \emph{distance-regular graphs}, which include the Johnson graphs $J(n,k)$, the Hamming graphs, the Kneser graphs, the odd graphs $O_k$, and the hypercube $Q_n$. Their spectra are completely determined by their intersection arrays, and a rich combinatorial theory has been developed for them; see Brouwer, Cohen, and Neumaier~\cite{Brower-1} for a comprehensive treatment. In recent years, considerable attention has shifted to graphs derived from these classical families---their line graphs, square graphs, subdivision graphs, and bipartite doubles---and to determining the spectra of these derived graphs.

\subsection*{Integral graphs}

A graph $\Gamma$ is called \emph{integral} if all eigenvalues of its adjacency matrix are integers. The concept was introduced by Harary and Schwenk~\cite{Harary-1} in 1974, who raised the question of characterizing all integral graphs. Despite several decades of effort, a complete classification remains elusive, and the search for new infinite families of integral graphs is an active line of research.

Watanabe and Schwenk in ~\cite{Watanabe} proved that a tree is integral if and only if it belongs to a specific finite set. Bussemaker, Cvetković ~\cite{Bussemaker}, and 
Schwenk ~\cite{Schwenk} independently classified cubic integral graphs, finding exactly 13 non-isomorphic examples. The study was extended to 4-regular graphs by Cvetković, Simić, and Stevanović in ~\cite{D.Cvetkovć}, who identified 1998 possible spectra for 4-regular bipartite integral graphs; see also the work of Ahmadi et al.~\cite{Ahmadi-1}. A significant source of infinite families of integral graphs is furnished by Cayley graphs: Abdollahi and Vatandoost~\cite{Abdollahi-1} characterized which Cayley graphs over abelian groups are integral, and subsequent work has produced integral Cayley graphs over many non-abelian groups. We refer to the monograph~\cite{Brower-2} and the references therein for a broader survey.

\subsection*{The Boolean lattice and its layer graphs}

The \emph{Boolean lattice} $BL_n$, $n\geq 3$, is the graph whose vertex set is the collection of all subsets of $[n]=\{1,2,\ldots,n\}$, where two subsets $U$ and $W$ are adjacent if and only if their symmetric difference has precisely one element. In other words, $BL_n$ is the $n$-dimensional hypercube $Q_n$, where vertices are identified with binary strings of length $n$ via characteristic vectors of subsets. In the graph $BL_n$, the \emph{layer} $L_k$ is the family of all $k$-element subsets of $[n]$. The induced subgraph $BL_n(k-1,k)$ on layers $L_{k-1}$ and $L_{k}$ is a bipartite graph with vertex bipartition $V_1 \cup V_2$, where
\[
V_1 = \{ U \subset [n] \mid |U| = k-1 \}, \quad |V_1| = \binom{n}{k-1},
\]
and
\[
V_2 = \{ W \subset [n] \mid |W| = k \}, \quad |V_2| = \binom{n}{k}.
\]
Each vertex in $V_1$ has degree $n-k+1$, and each vertex in $V_2$ has degree $k$, giving
\[
|V(BL_n(k-1,k))| = \binom{n+1}{k}, \qquad |E(BL_n(k-1,k))| = k\binom{n}{k}.
\]

These bipartite layer graphs enjoy rich structural properties. By~\cite{pap-sm-2}, the graph $BL_n(k-1,k)$ is isomorphic to the consecutive-layer graph $Q_n(k-1,k)$ of the hypercube. Furthermore, as shown by Cao, Lv, and Wang~\cite{M.Cao}, for $n\geq k$ the graph $BL_n(k-1,k)$ is isomorphic to the \emph{doubled Johnson graph} $J(n,k-1,k)$.

The graph $BL_n(k-1,k)$ is regular if and only if $n = 2k-1$. In this special case, every vertex has degree $k$, and the graph is isomorphic to the \emph{bipartite double cover} $2{\cdot}O_k$ of the odd graph $O_k$, a well-known distance-transitive graph (see~\cite{N.Biggs-1}). Consequently, when $n = 2k-1$, the line graph $L(BL_n(k-1,k))$ is vertex-transitive with degree $2k-2$, and its vertex and edge counts are
\[
|V(L(BL_n(k-1,k)))| = k\binom{2k-1}{k-1}, \qquad |E(L(BL_n(k-1,k)))| = k(k-1)\binom{2k-1}{k-1}.
\]

\subsection*{Line graphs and their spectra}

Recall that the \emph{line graph} $L(\Gamma)$ of a graph $\Gamma$ is the graph whose vertices are the edges of $\Gamma$, two vertices of $L(\Gamma)$ being adjacent whenever the corresponding edges of $\Gamma$ share an endpoint. A fundamental relationship between a graph and its line graph is given by the identity $B^T B = A(L(\Gamma)) + 2I$, where $B$ is the incidence matrix of $\Gamma$; this makes it possible to read off the spectrum of $L(\Gamma)$ from the singular values of $B$. Line graphs of bipartite graphs always have $-2$ as an eigenvalue, and by the Perron--Frobenius theorem, the largest eigenvalue of $L(\Gamma)$ equals $\Delta(\Gamma) + \delta(\Gamma) - 2$ when $\Gamma$ is biregular.

The spectra of line graphs of classical combinatorial graphs have attracted sustained attention. The line graph of the complete graph $K_n$ is the triangular graph $T(n)$, whose spectrum is classical. Mirafzal~\cite{pap-sm-1} introduced the family $H(n) = BL_n(1,2)$ and showed that its line graph $L(n)$ is a vertex-transitive integral graph with exactly five distinct eigenvalues $n-1, n-2, 0, -1, -2$; however, the multiplicities of these eigenvalues were left undetermined in~\cite{pap-sm-1}. Mirafzal also studied algebraic properties of line graphs arising from consecutive layers of the hypercube in~\cite{pap-sm-2}, where it is shown that these line graphs are Cayley graphs under certain conditions. More recently, Mirafzal~\cite{Mirafzal-crown} established that the line graph of the crown graph is distance-integral, and Kogani and Mirafzal~\cite{Kogani-Mirafzal} determined the distance spectrum of a broader class of distance-integral graphs. These results reflect a general trend of deriving spectral properties of derived graphs from those of well-understood base graphs.

\subsection*{Contribution of this paper}

In this paper, we determine the complete adjacency spectrum of the line graph $L(BL_n(k-1,k))$ for all $n \geq 2k-1$. Since $BL_n(k-1,k) \cong BL_n(n-k,n-k+1)$, we may assume without loss of generality that $n \geq 2k-1$. Our approach rests on two key observations: (i) the spectrum of $BL_n(k-1,k)$ can be deduced from that of the Johnson graph $J(n,k-1)$, and (ii) the spectrum of the line graph $L(BL_n(k-1,k))$ can then be extracted from the nonzero singular values of the incidence matrix $B$ of $BL_n(k-1,k)$ via Lemma~\ref{b.1}.

More precisely, we obtain the following. When $n \geq 2k$, the eigenvalues of $L(BL_n(k-1,k))$ are integers and their multiplicities are expressed in terms of binomial coefficients (Theorem~\ref{c.2}). When $n = 2k-1$, the graph $BL_n(k-1,k)$ is itself a distance-transitive regular graph, and its line graph is vertex-transitive with an entirely integral spectrum; we determine this spectrum completely in Theorem~\ref{c.3}. As a special case (Example~\ref{c.2.1}), we recover and complete the spectrum of $L(n)$ left open by Mirafzal in~\cite{pap-sm-1}: the multiplicities of the distinct eigenvalues $n-1, n-2, 0, -1, -2$ of $L(n)$ are $1, n-1, \tfrac{n(n-3)}{2}, n-1$, and $\tfrac{(n-1)(n-2)}{2}$, respectively.

\subsection*{Notation and adjacency spectra}

The \emph{adjacency matrix} $A(\Gamma)$ of a graph $\Gamma$ is the symmetric $\{0,1\}$-matrix whose rows and columns are indexed by $V(\Gamma)$, with the $(i,j)$-entry equal to 1 if $\{i,j\} \in E(\Gamma)$ and 0 otherwise. The \emph{eigenvalues} of $\Gamma$ are the roots of the characteristic polynomial $\det(xI_n - A(\Gamma))$. The \emph{spectrum} of $\Gamma$ is the multiset of eigenvalues of $A(\Gamma)$. When the distinct eigenvalues are $\lambda_1 > \lambda_2 > \cdots > \lambda_r$ with respective multiplicities $m_1, m_2, \ldots, m_r$, we write
\[
\mathrm{Spec}(\Gamma) = \bigl\{\lambda_1^{m_1},\, \lambda_2^{m_2},\, \ldots,\, \lambda_r^{m_r}\bigr\}.
\]
The graph $\Gamma$ is called \emph{integral} if all eigenvalues belong to $\mathbb{Z}$. Integral graphs form a distinguished class with applications in coding theory, quantum walks, and network design; see~\cite{Abdollahi-1,Ahmadi-1,Harary-1} and the references therein for recent developments.

\section{Definitions and Preliminaries}
All graphs in this paper are assumed to be finite, simple, and connected. We follow standard notation from~\cite{pap-cg-1} for concepts not defined here.
The \emph{incidence matrix} $B$ of a graph $\Gamma$ is the $\{0,1\}$-matrix with rows indexed by vertices and columns indexed by edges, such that the $(i,j)$-entry equals $1$ if and only if vertex $i$ is an endpoint of edge $j$. If $\Gamma$ has $n$ vertices and $e$ edges, then $B$ has order $n \times e$.

The rank of the incidence matrix has a clean combinatorial description: if $\Gamma$ is a graph with $n$ vertices and $c_0$ bipartite connected components, then $\mathrm{rank}(B) = n - c_0$. In particular, if $\Gamma$ is a connected bipartite graph with $n$ vertices, then $\mathrm{rank}(B) = n - 1$.

The \emph{line graph} of a graph $\Gamma$ is the graph $L(\Gamma)$ whose vertex set is $E(\Gamma)$, with two edges of $\Gamma$ being adjacent in $L(\Gamma)$ if and only if they share an endpoint in $\Gamma$.

\begin{lemma}[\cite{pap-cg-1}]\label{b.1} 
Let $B$ be the incidence matrix of the graph $\Gamma$ on $n$ vertices and $m$ edges, and let $L$ be
the line graph of $\Gamma$. Then $B^T B =  A(L)+2I_{m}$,  where $A(L)$ is the adjacency matrix of the line graph of $\Gamma$.
\end{lemma}

\begin{lemma}[\cite{pap-cg-1}]\label{b.2} 
Let $B$ be the incidence matrix of the graph $\Gamma$ on $n$ vertices and $m$ edges. Then $BB^T  = \Delta(\Gamma) +A(\Gamma)$, where $\Delta(\Gamma)$ is the diagonal 
$n \times n$ matrix whose $ii$-entry is equal to the degree of vertex $i$.
\end{lemma}

\begin{theorem}[\cite{Brower-1}]\label{b.3} 
The Johnson graph $J(n,k-1)$ has eigenvalues
\[
\lambda_j = (k-1-j)(n-k+1-j)-j,
\]
for $j=0,1,\dots,k-1$, with multiplicities
\[
m_j = \binom{n}{j} - \binom{n}{j-1}.
\]
\end{theorem}

\section{The Spectrum of the Line Graph of $BL_n(k-1,k)$}

In this section we determine the adjacency spectrum of the line graph $L(BL_n(k-1,k))$. We first compute the spectrum of $BL_n(k-1,k)$ itself (Theorem~\ref{c.1}), and then use Lemma~\ref{b.1} to derive the spectrum of its line graph in the two cases $n \geq 2k$ (Theorem~\ref{c.2}) and $n = 2k-1$ (Theorem~\ref{c.3}).

\begin{theorem}\label{c.1}
If $n \geq 2k-1$, is a fixed positive integer, then $BL_n(k-1,k)$ has eigenvalues 
\begin{itemize}
\item  $\mu_j = \pm \sqrt{(k-j)(n-k+1-j)},$ for $j=0,1,2,\ldots,k-1$, with multiplicities $m(\mu_j)$,
\item 0, with multiplicity $m(0) = \binom{n}{k} - \binom{n}{k-1}$,
\end{itemize}
where $m(\mu_j) = \binom{n}{j} - \binom{n}{j-1}$.
\end{theorem}
\begin{proof}
Let $G=BL_n(k-1,k)$ and $[n]=\{1, 2, \ldots, n\}$.
Recall that $V(G)=V_1\cup V_2$, where
\[
V_1 = \{ U \subset [n] \mid |U| = k-1 \}, \quad |V_1| = \binom{n}{k-1},
\]
and
\[
V_2 = \{ W \subset [n] \mid |W| = k \}, \quad |V_2| = \binom{n}{k},
\]
is defined already. Hence, the valency of  each vertex in $V_1$ is $n-k+1$, and  the valency for each vertex in $V_2$ is $k$. Then 
\[
|V(BL_n(k-1,k))| = \binom{n+1}{k}, \quad |E(BL_n(k-1,k))| = k\binom{n}{k}.
\]
Now, let $B$ be the incidence matrix of the graph $G$,  $\Delta(G)$ be the diagonal degree matrix of graph $G$, and $A(G)$ be the adjacency matrix of  graph $G$. Since $G$ is bipartite, then
\[
A(G) = \begin{pmatrix}
0 & N \\
N^{T} & 0
\end{pmatrix},
\]
where $N$ is the $\binom{n}{k-1} \times \binom{n}{k}$ bipartite incidence matrix so that rows correspond to vertices in the first part $k-1$-subsets of $G$, and columns correspond to vertices in the second part $k$-subsets of $G$.
Since $rank(N)= rank (N^{T})$, then $rank(A(G))= 2rank (N)$. Also, after ordering vertices of $G$ so that the $(k-1)$-subsets come first and the $k$-subsets second, we have
\[
\Delta(BL_n(k-1,k)) = 
\begin{pmatrix}
(n-k+1)I_{\binom{n}{k-1}} & 0 \\
0 & kI_{\binom{n}{k}}
\end{pmatrix}.
\]
Hence based on Lemma \ref{b.2},
\[
BB^{T} = A(G) +\begin{pmatrix}
(n-k+1)I_{\binom{n}{k-1}} & 0 \\
0 & kI_{\binom{n}{k}}
\end{pmatrix}.
\]
On the other hand, 
\[
{A(G)}^2 = \begin{pmatrix}
NN^{T} & 0 \\
0 & N^{T}N
\end{pmatrix},
\]
and it is well known that the squares of the nonzero eigenvalues of $A(G)$ are exactly the nonzero eigenvalues of $NN^{T}$.  
Therefore, we compute $NN^{T}$. We know that  $NN^{T}$ is  $\binom{n}{k-1} \times \binom{n}{k-1}$ matrix. 
Now let $X$ and $Z$ be  $k-1$ subset of $V_1$. If $|X\cap Z|=t$, then  $0\leq t\leq k-1$. Hence, if $t=k-1$ then $X=Z$, and since, each vertex of $V_1$ in graph $G$ is adjacent to exactly $n-k+1$, vertices in $V_2$,  
then
\[
(NN^{T})_{XX} = n-k+1.
\]
In particular,  if $t=k-2$ then $X\neq Z$, and since, there is exactly one $k$-set of $V_2$ so that is adjacent to $k-1$ subsets  $X$ and $Z$ of $V_1$, and hence
\[
(NN^{T})_{XZ} = 1.
\]
Especially, if $t\leq k-3$ then $X\neq Z$, and since, there is no $k$-set of $V_2$ so that is adjacent to $k-1$ subsets  $X$ and $Z$ of $V_1$, and hence
\[
(NN^{T})_{XZ} = 0.
\]
Thus,  if we consider the adjacency matrix $A_{J}$ of $J(n, k-1)$, then we can verify that
\[
NN^T = (n-k+1)I_{\binom{n}{k-1}} + A_{J},
\]
Moreover, based on Theorem \ref{b.3}, Johnson graph $J(n,k-1)$ has eigenvalues
\[
\lambda_j = (k-1-j)(n-k+1-j)-j,
\]
for $j=0,1,\dots,k-1$, with multiplicities
\[
m(\lambda_j) = \binom{n}{j} - \binom{n}{j-1}.
\]
Hence, $NN^{T}$ has nonzero eigenvalues
\[
\theta_j =(n-k+1)+\lambda_j=(k-j)(n-k+1-j),
\]
for $j=0,1,\dots,k-1$, with the same multiplicities $m(\theta_{j})=m(\lambda_{j})$. Thus,  $A(G)$  has nonzero eigenvalues 
\[
\mu_j = \pm \sqrt{(k-j)(n-k+1-j)},
\]
for  $j=0,1,2,\ldots,k-1$, with the same multiplicities $m(\mu_j)=m(\lambda_j)$. On the other hand, $NN^{T}$ is an invertible matrix, and hence $\mathrm{rank}(NN^{T})= \binom{n}{k-1}$, which implies $\mathrm{rank}(N)=\binom{n}{k-1}$. 
Therefore, $\mathrm{rank}(A(G))=2\binom{n}{k-1}$, and since
for $k\geq2$ we have $\binom{n+1}{k}>2\binom{n}{k-1}$, it follows that $0$ is an eigenvalue of $A(G)$ with multiplicity 
\[
m(0) = \binom{n+1}{k} -2\binom{n}{k-1} =  \binom{n}{k} - \binom{n}{k-1}. 
\]
Note that if $n=2k-1$, then $\binom{n}{k} - \binom{n}{k-1}=0$, and hence in this case $0$ is not an eigenvalue of $BL_n(k-1,k)$.
\end{proof}
\begin{theorem}\label{c.2}
If $k\geq2$ is a positive integer and $n\geq2k$, then the line graph $L(BL_n(k-1,k))$ has eigenvalues 
\begin{itemize}
\item $ -2$ with multiplicity $k\binom{n}{k} - \binom{n+1}{k} + 1$,
\item  $\mu_j = j-2$, for $j=1,2,\ldots,k-1$, with the multiplicities  $m(\mu_j)$,
\item  $\mu_j = n-1-j$, for $j=0,1,2,\ldots,k-1$, with the multiplicities  $m(\mu_j)$,
\item $ k-2$ with multiplicity $\binom{n}{k} - \binom{n}{k-1}$,
\end{itemize}
where
\[
m(\mu_j) = \binom{n}{j} - \binom{n}{j-1}.
\] 
\end{theorem}
\begin{proof}
Let $G=BL_n(k-1,k)$ be a subgraph of the Boolean lattice $BL_n$. Also, let $A(G)$ be the adjacency matrix of graph $G$, and let $B$ be the incidence matrix of the graph $G$ on $\binom{n+1}{k}$ vertices and $k\binom{n}{k}$ edges. 
Since $G$ is a connected bipartite graph, we have
 \[
\mathrm{rank}(B) = |V(G)| - 1 = \binom{n+1}{k}-1. 
\]
It is well known that $\mathrm{rank}(B^TB) = \mathrm{rank}(B)$ and $B^TB$ is a ${k\binom{n}{k}\times k\binom{n}{k}}$ matrix; since
for $k\geq2$ we have $k\binom{n}{k}>\binom{n+1}{k}-1$, it follows that $0$ is an eigenvalue of $B^TB$ with multiplicity $k\binom{n}{k} - \binom{n+1}{k} + 1$.
Thus, by Lemma~\ref{b.1}, $-2$ is an eigenvalue of the line graph $L(BL_n(k-1,k))$ with multiplicity $k\binom{n}{k} - \binom{n+1}{k} + 1$. 
On the other hand, by Lemma~\ref{b.2}, $BB^T =\Delta(G) +A(G)$, where 
\[
\Delta(G) = 
\begin{pmatrix}
(n-k+1)I_{\binom{n}{k-1}} & 0 \\
0 & kI_{\binom{n}{k}}
\end{pmatrix},
\]
and 
\[
A(G) = \begin{pmatrix}
0 & N \\
N^{T} & 0
\end{pmatrix},
\]
such that $N$ is the $\binom{n}{k-1} \times \binom{n}{k}$ bipartite incidence matrix so that rows correspond to vertices in the first part $k-1$-subsets of $G$, and columns correspond to vertices in the second part $k$-subsets of $G$. Based on previous Theorem we can verify that  $NN^{T}$ is an invertible matrix and it has nonzero eigenvalues
\[
\theta_j =(k-j)(n-k+1-j),
\]
for $j=0,1,\dots,k-1$, with  multiplicities 
\[
m(\theta_j) = \binom{n}{j} - \binom{n}{j-1}.
\] 
Also, it is not hard to see that if $\gamma$ is an eigenvalue of $BB^T$, where $\gamma\neq k$ then $(\gamma - k)(\gamma - (n - k + 1)$ is an eigenvalue of $NN^T$. Hence by solving the equation 
\[
(\gamma - k)(\gamma - (n - k + 1)=(k-j)(n-k+1-j),
\]
we have  $\gamma=j$ and $\gamma=n-j+1$ for $j=0,1,\ldots,k-1$, and hence, $BB^T$ has eigenvalues  $\gamma=j$ and $\gamma=n-j+1$  for $j=0,1,\ldots,k-1$   with the same  multiplicities $m(j)=m(n-j+1)=m(\theta_j)$. In particular, $\gamma=k$ is an eigenvalue of $BB^T$ with  multiplicity $\binom{n}{k} - \binom{n}{k-1}$. Also, we can verify that  all nonzero eigenvalues of $BB^T$ and $B^TB$ are identical  with the same multiplicities. Hence, for $j=1,\ldots,k-1$ $B^TB$ has eigenvalue $\gamma=j$ with the same  multiplicities $m(j)=m(\theta_j)$, also
for $j=0,1,\ldots,k-1$ $B^TB$ has eigenvalue $\gamma=n-j+1$ with the same  multiplicities $m(j)=m(\theta_j)$.
In particular, $\gamma=k$ is an eigenvalue of $B^TB$ with  multiplicity $\binom{n}{k} - \binom{n}{k-1}$.
Therefore, based on Lemma \ref{b.1},
 the line graph $L(BL_n(k-1,k))$ has eigenvalues 
\begin{itemize}
\item $ -2$ with multiplicity $k\binom{n}{k} - \binom{n+1}{k} + 1$,
\item  $\mu_j = j-2$, for $j=1,2,\ldots,k-1$, with the multiplicities  $m(\mu_j)$,
\item  $\mu_j = n-1-j$, for $j=0,1,2,\ldots,k-1$, with the multiplicities  $m(\mu_j)$,
\item $ k-2$ with multiplicity $\binom{n}{k} - \binom{n}{k-1}$,
\end{itemize}
where
\[
m(\mu_j) = \binom{n}{j} - \binom{n}{j-1}.
\] 
\end{proof}
\begin{example}\label{c.2.1}
Consider the graph $H(n)=BL_n(1,2)$. If $n\geq 4$ is a fixed positive integer, then the spectrum of the graph $L(n)$ is
\[
\mathrm{Spec}(L(n)) = \bigl\{ (n-1)^{1},\, (n-2)^{n-1},\, 0^{\frac{n(n-3)}{2}},\, (-1)^{n-1},\, (-2)^{\frac{(n-1)(n-2)}{2}}\bigr\}.
\]
In particular, this determines the multiplicities of all five distinct eigenvalues of $L(n)$, completing the result of~\cite{pap-sm-1}.
\end{example}
\begin{theorem}\label{c.3}
If $k\geq2$ is a positive integer and $n=2k-1$, then the line graph $L(BL_n(k-1,k))$ has eigenvalues 
\begin{itemize}
\item $ -2$, with multiplicity $ (k-2)\binom{2k-1}{k-1}+1$,
\item  $\mu_i = i-2$, for $i=1,2,\ldots,k-1$, with the multiplicities  $m(\mu_i)$,
\item $\mu_i = 2k-i-2$, for $i=0,1,2,\ldots,k-1$, with the multiplicities  $m(\mu_i)$,
\end{itemize}
where
\[
m(\mu_i) = \binom{2k-1}{i} - \binom{2k-1}{i-1}.
\] 
\end{theorem}
\begin{proof}
Let $G=BL_n(k-1,k)$ be a subgraph of the Boolean lattice $BL_n$. Recall that $V(G)=V_1\cup V_2$, where
\[
V_1 = \{ U\subset [n]  \mid |U|=k-1 \},
\]
and
\[
V_2 = \{  W\subset [n]  \mid |W|=k\},
\]
is defined already. Let $A(G)$ be the adjacency matrix of $G$, and let $B$ be the incidence matrix of $G$ on $2\binom{2k-1}{k-1}$ vertices and $k\binom{2k-1}{k-1}$ edges. Then $BB^T = \Delta(G) + A(G)$, where $\Delta(G)$ is the diagonal matrix of size $2\binom{2k-1}{k-1} \times 2\binom{2k-1}{k-1}$ with each diagonal entry equal to $k$ (the common degree of all vertices of $G$).
It is well known if $n=2k-1$ then the  graph $G$ (the adjacency matrix $A(G)$), has eigenvalues $\lambda_i=\pm(k-i)$ , for $i=0,1,2,\ldots,k-1$ with the multiplicities
\[
m(\lambda_i) = \binom{2k-1}{i} - \binom{2k-1}{i-1},
\] 
(see ~\cite{N.Biggs-1} Page 74). From the positive eigenvalues of $G$, if $\lambda_i=k-i>0$, then based on Lemma \ref{b.2}, $BB^T$ has eigenvalues 
$\theta_i=k+\lambda_i=2k-i$,
hence $\theta_0=2k, \theta_1=2k-1, \ldots,\theta_{k-1}=k+1$ are eigenvalues of $BB^T$  with the same multiplicities $m(\theta_i)= m(\lambda_i)$,
also; from the negative eigenvalues of $G$, if $\lambda_i=-(k-i)<0$, then based on Lemma \ref{b.2}, $BB^T$ has eigenvalues $\theta_i=k-\lambda_i=i$,  
and hence $\theta_0=0, \theta_1=1, \ldots,\theta_{k-1}=k-1$ are eigenvalues of $BB^T$  with the same multiplicities $m(\theta_i)= m(\lambda_i)$.
In particular, since $\lambda_0=-k$ is a negative eigenvalue of $G$ with multiplicity $1$, then $\theta_0=0$ is an eigenvalue of $BB^T$ with multiplicity $m(\theta_0)= m(\lambda_0)=1$.  Also, we can verify that  all nonzero eigenvalues of $BB^T$ and $B^TB$ are identical  with the same multiplicities.
Moreover, since $G$ is a connected bipartite graph, we have
\[
\mathrm{rank}(B) = |V(G)| - 1 = 2\binom{2k-1}{k-1}-1.
\]
It is well known that $\mathrm{rank}(B) = \mathrm{rank}(B^TB) = \mathrm{rank}(BB^T) = \mathrm{rank}(B^T)$. 
Since $B^TB$ is a $k\binom{2k-1}{k-1} \times k\binom{2k-1}{k-1}$ matrix, and since
for $k\geq2$ we have $k\binom{2k-1}{k-1} > 2\binom{2k-1}{k-1}-1$, it follows that $0$ is an eigenvalue of $B^TB$ with multiplicity 
\[
m(0) = k\binom{2k-1}{k-1} - \bigl(2\binom{2k-1}{k-1}-1\bigr) = (k-2)\binom{2k-1}{k-1}+1. 
\]
Thus, by Lemma~\ref{b.1}, $-2$ is an eigenvalue of the line graph $L(BL_n(k-1,k))$ with multiplicity $m(-2)=m(0)$. 
On the other hand, if $L(G)$ is the line graph of $G$, then by Lemma~\ref{b.1}, $B^T B =  A(L(G))+2I_{m}$, where $A(L(G))$ is the adjacency matrix of $L(G)$ and $m=k\binom{2k-1}{k-1}$. Hence, 
for $i=0, 1, 2,\ldots,k-1$, if $\lambda_i>0$, then $\mu_i=\theta_i-2=2k-i-2$ and hence $\mu_0=2k-2, \mu_1=2k-3, \ldots, \mu_{k-1}=k-1$ are eigenvalues of $L(G)$ with multiplicities $m(\mu_i)=m(\theta_i)$; for $i= 1, 2,\ldots,k-1$, if $\lambda_i<0$, then $\mu_i=\theta_i-2=i-2$ and hence $\mu_1=-1, \ldots, \mu_{k-1}=k-3$ are eigenvalues of $L(G)$ with multiplicities $m(\mu_i)=m(\theta_i)$.
\end{proof}

\section{Conclusion}

In this paper, we have determined the complete adjacency spectrum of the line graph $L(BL_n(k-1,k))$ for all integers $k \geq 2$ and $n \geq 2k-1$. Our main contributions are as follows.

\begin{itemize}
  \item In Theorem~\ref{c.1}, we computed the spectrum of the bipartite layer graph $BL_n(k-1,k)$ for $n \geq 2k-1$, expressing the eigenvalues and their multiplicities via those of the Johnson graph $J(n,k-1)$.
  \item In Theorem~\ref{c.2} (the case $n \geq 2k$), we showed that all eigenvalues of $L(BL_n(k-1,k))$ are integers, and we gave explicit formulas for eigenvalues and their multiplicities in terms of binomial coefficients.
  \item In Theorem~\ref{c.3} (the case $n = 2k-1$), where $BL_n(k-1,k) \cong 2{\cdot}O_k$, we established the complete integral spectrum of the vertex-transitive line graph $L(BL_n(k-1,k))$.
  \item As a special case (Example~\ref{c.2.1}), we completed the result of Mirafzal~\cite{pap-sm-1} by computing the multiplicities of all eigenvalues of the line graph $L(n) = L(BL_n(1,2))$, which were left undetermined in~\cite{pap-sm-1}.
\end{itemize}

These results show that the line graphs $L(BL_n(k-1,k))$ constitute a rich infinite family of integral graphs accessible through combinatorial methods. Several natural questions remain open. For instance, one may ask whether $L(BL_n(k-1,k))$ is determined by its spectrum---that is, whether a graph with the same spectrum must be isomorphic to $L(BL_n(k-1,k))$. Another direction is to determine the distance spectrum of these line graphs, extending the distance-integral results of Mirafzal~\cite{Mirafzal-crown} and Kogani--Mirafzal~\cite{Kogani-Mirafzal} to the present family. We hope that the techniques introduced here will be useful in these and related investigations.


\section*{Funding and Conflict of Interest} 
The authors have received no funding for this work and declare no conflict of interest.

\bigskip

\begin{thebibliography}{12}
\bibitem{Abdollahi-1}A. Abdollahi and E. Vatandoost, 
Which Cayley graphs are integral?, \emph{Electron. J. Combin.}, 
\textbf{16}(1) (2009), R122, 1--17.
\bibitem{Ahmadi-1}O. Ahmadi, N. Alon, I.\,F. Blake and I.\,E. Shparlinski, 
Graphs with integral spectrum, \emph{Linear Algebra Appl.},
\textbf{430} (2009), 547--552.
\bibitem{N.Biggs-1}N. Biggs, 
Some odd graph theory, \emph{Ann. New York Acad. Sci.},
\textbf{319} (1979), 71--81.
\bibitem{Brower-1}A.\,E. Brouwer, A.\,M. Cohen and A. Neumaier, 
\emph{Distance Regular Graphs}, Springer, Berlin, 1989.
\bibitem{Brower-2}A.\,E. Brouwer and W.\,H. Haemers, 
\emph{Spectra of Graphs}, Springer, New York, 2012.
\bibitem{Bussemaker}F. C. Bussemaker and D. M. Cvetković, 
There are exactly $13$ connected, cubic, integral graphs,
 \emph{Publikacije Elektrotehničkog Fakulteta. Serija Matematika i Fizika}, \textbf{544} (1976), 43--48.
\bibitem{M.Cao}M. Cao, B. Lv and K. Wang,
On even-cycle-free subgraphs of the doubled Johnson graph,
\emph{Appl. Math. Comput.}, \textbf{403} (2021), 125763.
\bibitem{D.Cvetkovć}D. Cvetković, S. Simić and D. Stevanović,  
4-regular integral graphs, \emph{Univ. Beograd. Publ. Elektrotehn. Fak}, Ser. Mat. \textbf{9} (1998), 109--123.
\bibitem{pap-cg-1}C. Godsil and G. Royle,
\emph{Algebraic Graph Theory}, Springer, New York, 2001.
\bibitem{Harary-1}F. Harary and A. Schwenk, 
Which graphs have integral spectra?, \emph{Lect. Notes Math.},
\textbf{406}, Springer, Berlin, (1974), 45--50.
\bibitem{Kogani-Mirafzal}R. Kogani and S.\,M. Mirafzal, 
On determining the distance spectrum of a class of distance integral graphs, \emph{J. Algebr. Syst.}, 
\textbf{10}(2) (2023), 299--308.
\bibitem{Mirafzal-crown}S.\,M. Mirafzal, 
The line graph of the crown graph is distance integral, \emph{Linear Multilinear Alg.}, \textbf{71}(4) (2023), 662--672.
\bibitem{pap-sm-1}S.\,M. Mirafzal,
A new class of integral graphs constructed from the hypercube,
\emph{Linear Algebra Appl.}, \textbf{558} (2018), 186--194.
\bibitem{pap-sm-2}S.\,M. Mirafzal,
Cayley properties of the line graphs induced by consecutive layers of the hypercube, \emph{Bull. Malays. Math. Sci. Soc.}, \textbf{44} (2021), 1309--1326.
\bibitem{Schwenk}A. J. Schwenk, 
Exactly thirteen connected cubic graphs have integral spectra. In: Alavi, Y., Lick, D.R. (eds) Theory and Applications of Graphs. 
Lecture Notes Math., vol 642. Springer, Berlin, Heidelberg. https://doi.org/10.1007/BFb0070407.
\bibitem{Watanabe}M. Watanabe and A. J. Schwenk, Integral starlike trees,  \emph{J. Austral. Math. Soc., Series A},  \textbf{28}(1) (1979), 120--128.
\end{thebibliography}
\end{document}